\documentclass[12pt]{article}
\usepackage{amsthm,amsmath,amssymb}
\usepackage{graphicx}
\usepackage[shortlabels]{enumitem}
\usepackage{multirow}
\usepackage{hyperref}
\usepackage{url}
\usepackage{longtable}
\usepackage{tikz}

\theoremstyle{plain}
\newtheorem{theorem}{Theorem}[section]

\newtheorem{proposition}[theorem]{Proposition}
\newtheorem{lemma}[theorem]{Lemma}
\newtheorem{corollary}[theorem]{Corollary}

\theoremstyle{definition}

\theoremstyle{remark}
\newtheorem{remark}[theorem]{Remark}
\newtheorem{example}{Example}

\title{Edge-regular graphs with regular cliques}

\author{Gary R. W. Greaves\thanks{Supported by  the Singapore
Ministry of Education Academic Research Fund (Tier 1);  grant number:  RG127/16.}\\
  School of Physical and Mathematical Sciences, \\
  Nanyang Technological University, \\
   21 Nanyang Link, Singapore 637371\\
  {\tt grwgrvs@gmail.com}
\and
  Jack H. Koolen\thanks{Partially supported by the National Natural Science Foundation of China
(No. 11471009 and No.11671376).}\\
  Wen-Tsun Wu Key Laboratory of CAS,\\
  School of Mathematical Sciences, \\
  University of Science and Technology of China, \\
  Hefei, Anhui, 230026, P.R. China\\
  {\tt koolen@ustc.edu.cn}
}

\begin{document}
	
\maketitle

\begin{abstract}
	We exhibit infinitely many examples of edge-regular graphs that have regular cliques and that are not strongly regular.
	This answers a question of Neumaier from 1981.
\end{abstract}

\section{Introduction}

In this paper, all graphs are finite, without loops or multiple edges.
For a graph $\Gamma$ and a vertex $v$ of $\Gamma$, we use $\Gamma(v)$ to denote the set of neighbours of $v$ in $\Gamma$.
A non-empty $k$-regular graph $\Gamma$ on $N$ vertices is called \textbf{edge-regular} if there exists a constant $\lambda$ such that every pair of adjacent vertices has precisely $\lambda$ common neighbours.
The quantities $N$, $k$, and $\lambda$ are called the \textbf{parameters} of $\Gamma$ and are usually written as the triple $(N,k,\lambda)$.

A \textbf{strongly regular graph} is defined to be an edge-regular graph with parameters $(N,k,\lambda)$ such that every pair of non-adjacent vertices has precisely $\mu$ common neighbours.
The \textbf{parameters} of a strongly regular graph are given by the quadruple $(N,k,\lambda,\mu)$.
A clique $\mathcal C$ is called \textbf{regular} if every vertex not in $ \mathcal C$ is adjacent to a constant number $e>0$ of vertices in $\mathcal C$.
The value $e$ is called the \textbf{nexus} and the clique $\mathcal C$ is called \textbf{$e$-regular}.

We are concerned with edge-regular graphs that have regular cliques.
Neumaier ~\cite[Corollary 2.4]{Neu:regCliques} showed that an edge-regular, vertex-transitive, edge-transitive graph that has a regular clique must be strongly regular.
%
He further posed the following question.

\textbf{Question~\cite[Page 248]{Neu:regCliques}:}
Is every edge-regular graph with a regular clique strongly regular?

We answer this question in the negative, exhibiting infinite families of edge-regular graphs that are not strongly regular and that have regular cliques (see Section~\ref{sec:conclusion}).
Our graphs are Cayley graphs and hence they are vertex transitive.
For background on Cayley graphs and vertex- and edge-transitivity, we refer the reader to Godsil and Royle's book~\cite{God01}.

We also point out that Goryainov and Shalaginov~\cite{Goryainov14:Deza} found four non-isomorphic examples of edge-regular graphs with parameters $(24,8,2)$ that are not strongly regular and each have a $1$-regular clique.

In Section~\ref{sec:construction}, we introduce a parameterised Cayley graph and determine some of its properties
In Section~\ref{sec:sufficient_conditions_for_edge_regularity} and Section~\ref{sec:notSRGs}, we focus on specific parameters for the Cayley graph and find conditions that guarantee the graph is edge-regular and not strongly regular, respectively.
Finally, in Section~\ref{sec:conclusion}, we give parameters for our Cayley graphs that produce infinite families of graphs pertinent to Neumaier's question. 

\section{A parametrised Cayley graph} 
\label{sec:construction}

First we fix some notation.
Denote by $\mathbb Z_r$, the ring of integers modulo $r$, and by $\mathbb F_q$, the finite field of $q$ elements (where $q$ is a prime power).
Given an additive group $\mathfrak G$, we use $\mathfrak G^*$ to denote the subset of elements of $\mathfrak G$ that are not equal to the identity.
Let $G_{l,m,q} = \mathbb Z_l \oplus \mathbb Z_2^m \oplus \mathbb F_q$ where $l$ and $m$ are positive integers, and $q$ is a prime power.
Fix a primitive element $\rho$ of $\mathbb F_q$.
Next we set up a generating set for $G_{l,m,q}$.
Define the subset $S_0$ of $G_{l,m,q}$ as 
\[
	S_0 := \{ (g,0) \;|\; g \in (\mathbb Z_l \oplus \mathbb Z_2^m)^* \}.
\]
Let $\pi : (\mathbb Z_2^m)^* \to \mathbb Z_{2^m-1}$ be a bijection.
For each $z \in (\mathbb Z_2^m)^*$, define
\[
	S_{z,\pi} := \{ (0, z, \rho^j) \; | \; j \equiv \pi(z) \pmod {2^m-1} \} \subset G_{l,m,q}.
\]
Now define $S(\pi) := S_0 \cup \bigcup_{z \in (\mathbb Z_2^m)^*} S_{z,\pi}$.
Observe that $S(\pi)$ is a generating set for $G_{l,m,q}$.
We consider the Cayley graph of the (additive) group $G_{l,m,q}$ with generating set $S(\pi)$, which we denote by $\operatorname{Cay}(G_{l,m,q},S(\pi))$.
This parametrised Cayley graph is the main object of this paper.

\begin{lemma}\label{lem:undirected}
	Let $l$ and $m$ be positive integers, let $\pi : (\mathbb Z_2^m)^* \to \mathbb Z_{2^m-1}$ be a bijection, set $n = 2^m-1$, and let $q \equiv 1 \pmod{2n}$ be a prime power.
	Then $\operatorname{Cay}(G_{l,m,q},S(\pi))$ is an undirected graph.
\end{lemma}
\begin{proof}
	Write $q = 2nr+1$ for some $r$.
	It suffices to show that the set $S$ is symmetric, that is, for all $g \in S$ we have $-g \in S$.
	It is easy to see that $S_0$ is symmetric.
	For each $z \in (\mathbb Z_2^m)^*$, the set $S_{z,\pi}$ is symmetric since $-1 = \rho^{nr}$.
	Indeed, suppose $(0,z,\rho^j) \in S_{z,\pi}$, then its inverse, $(0,z,-\rho^j) = (0,z,\rho^{j+nr})$ is also in $S_{z,\pi}$.
\end{proof}

Fix a prime power, $q$, and positive integers $l$ and $m$.
For each $f \in \mathbb F_q$, define the set $\mathcal C_f := \{ (g, f) : g \in \mathbb Z_l \oplus \mathbb Z_2^m \} \subset G_{l,m,q}$.
A set of cliques of a graph $\Gamma$ that partition the vertex set of $\Gamma$ is called a \textbf{spread} in $\Gamma$.
Soicher~\cite{Soi:CAB15} studied edge-regular graphs that have spreads of regular cliques.
The next lemma shows that the graph $\operatorname{Cay}(G_{l,m,q},S(\pi))$ has a spread of regular cliques.

\begin{lemma}\label{lem:regClique}
	Let $l$ and $m$ be positive integers, let $\pi : (\mathbb Z_2^m)^* \to \mathbb Z_{2^m-1}$ be a bijection, and let $q$ be a prime power.
	Then the set $\{ \mathcal C_f : f \in \mathbb F_q \}$ is a spread of $1$-regular cliques in $\operatorname{Cay}(G_{l,m,q},S(\pi))$.
\end{lemma}
\begin{proof}
	Let $\Gamma = \operatorname{Cay}(G_{l,m,q},S(\pi))$.
	First, it is clear that the set $\{ \mathcal C_f : f \in \mathbb F_q \}$ is a partition of the vertex set of $\Gamma$.
	
	Fix an element $f \in \mathbb F_q$ and take two distinct elements $x$ and $y$ in $\mathcal C_f$.
	Since the difference $x - y$ is in $S_0$, we have that $x$ and $y$ are adjacent in $\Gamma$.
	
	Now we show that each vertex of $\Gamma$ outside $\mathcal C_f$ is adjacent to precisely one vertex of $\mathcal C_f$.
	Let $v = (g_1,g_2,g_3)$ 
	with $(g_1,g_2) \in \mathbb Z_l \oplus \mathbb Z_2^m$ and $g_3 \in \mathbb F_q$, 
	be a vertex not in the clique $\mathcal C_f$.
	Then $g_3 \ne f$.
	The vertex $v$ is adjacent to a vertex of the clique $\mathcal C_f$ via the unique generator $(s_1,s_2,s_3) \in S(\pi)$ with $s_3 = -g_3+f$.
	Hence each vertex not in $\mathcal C_f$ is adjacent to precisely one vertex of $\mathcal C_f$, as required.
\end{proof}

Next, given a finite field $\mathbb F_q$ with $q = 2n r+1$ and a fixed primitive element $\rho$, define, for each $i \in \{0,\dots,n-1\}$, the $n$th \textbf{cyclotomic class} $C^n_q(i)$ of $\mathbb F_q$ as
\[
	C^n_q(i) := \{ \rho^{nj+i} \; | \; j \in \{ 0,\dots,2r-1 \} \}.
\]
For $a, b \in \{0,\dots,n-1\}$, the $n$th \textbf{cyclotomic number} $c^n_q(a,b)$ is defined as $c^n_q(a,b) := |(C^n_q(a)+1)\cap C^n_q(b)|$.
We refer to \cite{MacWilliams72} for background on cyclotomic numbers.
 
%

Our next result is about the number of common neighbours of two adjacent vertices of $\operatorname{Cay}(G_{l,m,q},S(\pi))$.

\begin{lemma}\label{lem:valencies}
	Let $l$ and $m$ be positive integers, let $\pi : (\mathbb Z_2^m)^* \to \mathbb Z_{2^m-1}$ be a bijection, and let $q$ be a prime power.
	Let $\Gamma = \operatorname{Cay}(G_{l,m,q},S(\pi))$ and let $v$ be a vertex of $\Gamma$.
	Then the subgraph of $\Gamma$ induced on $\Gamma(v)$ has valencies $2^m l-2$ and 
	$$
		\sum_{\substack{h \in (\mathbb Z^m_2)^* \\ h \ne g}}c^n_q(\pi(h-g)-\pi(g),\pi(h)-\pi(g)), \text{ for each } g \in (\mathbb Z^m_2)^*.
	$$
\end{lemma}

\begin{proof}
	Since $\Gamma$ is vertex transitive, we can assume that $v$ is the identity of $G_{l,m,q}$.
	The neighbours of $v$ are then the elements of $S(\pi)$.
	Recall that $S(\pi) = S_0 \cup \bigcup_{z \in (\mathbb Z_2^m)^*} S_{z,\pi}$.
	Each element of the set $S_0$ is adjacent to every other element of $S_0$ and no elements of $S(\pi) \backslash S_0$.
	Since $S_0$ has cardinality $2^m l-1$, we have the first part of the lemma.
	
	Take an element $s \in S(\pi) \backslash S_0$.
	Then $s = (0,g,\rho^{j})$ for some $g \in (\mathbb Z_2^m)^*$ and $j \equiv \pi(g) \pmod{2^m-1}$.
	For each $h \in (\mathbb Z^m_2)^*$ with $h \ne g$, the number of neighbours of $s$ in $S_{h,\pi}$ is
	\[
		c^n_q(\pi(h-g)-\pi(g),\pi(h)-\pi(g))
	\]
	as required.
\end{proof}

The final result in this section provides a formula that we will use to count the number of common neighbours of two nonadjacent vertices.

\begin{proposition}\label{pro:mu}
	Let $l$ and $m$ be positive integers, let $\pi : (\mathbb Z_2^m)^* \to \mathbb Z_{2^m-1}$ be a bijection, and let $q$ be a prime power.
	Let $\Gamma = \operatorname{Cay}(G_{l,m,q},S(\pi))$ and let $v$ and $w$ be the vertices of $\Gamma$ corresponding to $(0,0,0) \in G_{l,m,q}$ and $(0,g,\rho) \in G_{l,m,q}$, respectively, where $\pi(g) \ne 1$.
	Then the number of common neighbours of $v$ and $w$ is
	\[
		2 + \sum_{\substack{h \in (\mathbb Z_2^m)^* \\ h \ne g}} c_q^{2^m-1}(\pi(h)-1,\pi(h+g)-1).
	\]
\end{proposition}
\begin{proof}
	The neighbours of $v$ are precisely the elements of $S(\pi)$.  
	The vertices $v$ and $w$ have precisely $1$ common neighbour in $S_0$ and precisely $1$ common neighbour in $\{ (g+h,\rho) \;|\; h \in (\mathbb Z_l \oplus \mathbb Z_2^m)^* \}$.
	The remaining common neighbours of $v$ and $w$ are elements of the form $(0,z,\rho^i)$ such that $z=g+h$ for some $h \in (\mathbb Z_2^m)^*\backslash \{ g \}$, and $i \equiv \pi(z) \pmod {2^m -1}$, $j \equiv \pi(h) \pmod {2^m -1}$, and $\rho^i \equiv \rho+\rho^j$.
	The number of such elements is equal to
	\[
		\sum_{\substack{h \in (\mathbb Z_2^m)^* \\ h \ne g}} c_q^{2^m-1}(\pi(h)-1,\pi(h+g)-1).
	\]
\end{proof}

%
%



\section{Sufficient conditions for edge-regularity} 
\label{sec:sufficient_conditions_for_edge_regularity}

Next we find sufficient conditions to guarantee that the graph $\operatorname{Cay}(G_{l,m,q},S(\pi))$ is edge-regular for $m=3$ and $m=2$.

\subsection{Edge-regularity for the graphs $\operatorname{Cay}(G_{l,3,q},S(\pi))$} 
\label{sub:m_3}

Here we focus on $m=3$.
We will need the following equalities for the $7$th cyclotomic numbers.

\begin{lemma}[See \cite{MacWilliams72}]\label{lem:cycEq7}
	Let $q \equiv 1 \pmod{14}$ be a prime power.
	Then
	\begin{enumerate}[(i)]
		\item $c^7_q(1,3) = c^7_q(6,2) = c^7_q(5,4)$,
		\item $c^7_q(1,5) = c^7_q(6,4) = c^7_q(3,2)$, and
		\item $c^7_q(a,b) = c^7_q(b,a)$ for all $a,b \in \{0,\dots,6\}$. 
	\end{enumerate}
\end{lemma}

Elements of $\mathbb Z_2$ are naturally identified with elements of the set $\{ 0,1 \} \subset \mathbb Z$; let $\hat{x}$ denote the image in $\{0,1\}$ of an element $x \in \mathbb Z_2$.
Let $\phi  : (\mathbb Z^3_2)^* \to \mathbb Z_7$ be given by $\phi (x_2,x_1,x_0) := \hat{x_0}+2\hat{x_1}+4\hat{x_2} \pmod 7$ and let $\operatorname{wt} : (\mathbb Z^3_2) \to \mathbb Z$  be given by $\operatorname{wt} (x_2,x_1,x_0) := \hat{x_0}+\hat{x_1}+\hat{x_2}$.
Define index-shifting functions $\sigma_{+}, \sigma_{-} :  (\mathbb Z^3_2) \to (\mathbb Z^3_2)$ as
$\sigma_{+} (x_2,x_1,x_0) := (x_1,x_0,x_2)$ and $\sigma_{-} (x_2,x_1,x_0) := (x_0,x_2,x_1)$.

Now we define the functions $\Psi_1, \Psi_2 : (\mathbb Z^3_2)^* \to  \mathbb Z_7$ as
\begin{align*}
	\Psi_1(\mathbf x) &= \begin{cases}
		\phi( \sigma_+(\mathbf x) ), & \text{ if } \operatorname{wt}(\mathbf x) \text{ is odd,} \\
		\phi( \sigma_-(\mathbf x) ), & \text{ if } \operatorname{wt}(\mathbf x) \text{ is even;} 
	\end{cases}  \\
	\Psi_2(\mathbf x) &= \begin{cases}
		\phi( \sigma_+(\mathbf x) + \mathbf x ), & \text{ if } \operatorname{wt}(\mathbf x) \text{ is odd,} \\
		\phi( (1,1,1) + \mathbf x ), & \text{ if } \operatorname{wt}(\mathbf x) \text{ is even.} 
	\end{cases}
\end{align*}

We use the functions $\Psi_1$ and $\Psi_2$ since they have the following nice property, which follows from Lemma~\ref{lem:cycEq7}.

\begin{proposition}\label{pro:sameCycNumbers7}
	Let $q \equiv 1 \pmod{14}$ be a prime power and let $g,h \in (\mathbb Z^3_2)^*$ with $g \ne h$.
	Then 
	\begin{enumerate}[(i)]
		\item $c^7_q(\Psi_1(h-g)-\Psi_1(g),\Psi_1(h)-\Psi_1(g)) = c^7_q(1,5)$;
		\item $c^7_q(\Psi_2(h-g)-\Psi_2(g),\Psi_2(h)-\Psi_2(g)) = c^7_q(1,3)$.
	\end{enumerate}
\end{proposition}

Now we give sufficient conditions that guarantee that $\operatorname{Cay}(G_{l,3,q},S(\pi))$ is edge-regular.

\begin{theorem}\label{thm:ERG3}
	Let $q \equiv 1 \pmod{14}$ be a prime power, let $c \equiv 1 \pmod 4$, and let $l = (3c+1)/4$.
	Suppose that $c = c^7_q(1,5)$ (resp. $c =c^7_q(1,3)$) and set $\Gamma = \operatorname{Cay}(G_{l,3,q},S(\pi))$ where $\pi = \Psi_1$ (resp. $\pi = \Psi_2$).
	Then $\Gamma$ is edge-regular with parameters $(8lq,8l-2+q,8l-2)$ having a regular clique of order $8l$.
\end{theorem}
\begin{proof}
	Let $v$ be a vertex of $\Gamma$.
	By Lemma~\ref{lem:valencies}, the subgraph $\Delta$ induced on $\Gamma(v)$ has valencies $8l-2$ and \[
		\sum_{\substack{h \in (\mathbb Z^3_2)^* \\ h \ne g}}c^7_q(\pi(h-g)-\pi(g),\pi(h)-\pi(g)), \text{ for each } g \in (\mathbb Z^3_2)^*.
	\]
	Suppose $c = c^7_q(1,5)$ (resp. $c = c^7_q(1,3)$).
	Then, by Proposition~\ref{pro:sameCycNumbers7}, the graph $\Delta$ is regular with valency $8l-2 = 6c^7_q(1,5)$ (resp. $8l-2 = 6c^7_q(1,3)$).
	By Lemma~\ref{lem:regClique}, the graph $\Gamma$ has a regular clique, as required.
\end{proof}

\subsection{Edge-regularity for the graphs $\operatorname{Cay}(G_{l,2,q},S(\pi))$} 
\label{sub:_m_2}

Here we focus on $m=2$.
We will need the following equality for the $3$rd cyclotomic numbers.

\begin{lemma}[See \cite{MacWilliams72}]\label{lem:cycEq3}
	Let $q \equiv 1 \pmod 6$ be a prime power.
	Then $c^3_q(1,2) = c^3_q(2,1) = c^3_q(0,0)+1$.
\end{lemma}

As a contrast to the previous subsection, in the next result we show that for $m=2$, the choice of bijection $\pi$ in the graph $\operatorname{Cay}(G_{l,2,q},S(\pi))$ is not important.

\begin{proposition}\label{pro:goodBijl2}
	Let $\pi : (\mathbb Z^2_2)^* \to  \mathbb Z_3$ be a bijection and let $g,h \in (\mathbb Z^2_2)^*$ with $g \ne h$.
	Suppose $q \equiv 1 \pmod{6}$ is a prime power. 
	Then we have the equality $c^3_q(\pi(h-g)-\pi(g),\pi(h)-\pi(g)) = c_q^3(1,2)$.
\end{proposition}
\begin{proof}
	Observe that $\pi(h-g)$, $\pi(g)$, and $\pi(h)$ are pairwise distinct.
	Hence $(\pi(h-g)-\pi(g),\pi(h)-\pi(g))$ is either $(1,2)$ or $(2,1)$ and, using Lemma~\ref{lem:cycEq3}, we see that $c^3_q(\pi(h-g)-\pi(g),\pi(h)-\pi(g)) = c_q^3(1,2)$.
\end{proof}

Now we give sufficient conditions that guarantee that $\operatorname{Cay}(G_{l,2,q},S(\pi))$ is edge-regular.

\begin{theorem}\label{thm:ERG2}
	Let $q \equiv 1 \pmod{6}$ be a prime power and let $\pi : (\mathbb Z^2_2)^* \to  \mathbb Z_3$ be a bijection.
	Suppose $c = c_q^3(1,2)$ is odd and $l = (c+1)/2$.
	Then $\operatorname{Cay}(G_{l,2,q},S(\pi))$ is edge-regular with parameters $(4lq,4l-2+q,4l-2)$ having a regular clique of order $4l$.
\end{theorem}
\begin{proof}
	Let $v$ be a vertex of $\Gamma$.
	By Lemma~\ref{lem:valencies}, the subgraph $\Delta$ induced on $\Gamma(v)$ has valencies $4k-2$ and \[
		\sum_{\substack{h \in (\mathbb Z^2_2)^* \\ h \ne g}}c^3_q(\pi(h-g)-\pi(g),\pi(h)-\pi(g)), \text{ for each } g \in (\mathbb Z^3_2)^*.
	\]
	By Proposition~\ref{pro:goodBijl2}, the graph $\Delta$ is regular with valency $4l-2 = 2c^3_q(1,2)$.
	By Lemma~\ref{lem:regClique}, the graph $\Gamma$ has a regular clique, as required.
\end{proof}

By Theorem~\ref{thm:ERG2}, for a prime power $q \equiv 1 \pmod{6}$, all we need is that $c_q^3(1,2)$ is odd to ensure that we can construct an edge-regular graph having a regular clique.
Next we record a result to help us control the parity of $c_q^3(1,2)$.

\begin{lemma}[{\cite[Lemma 4]{Storer67}}]\label{lem:2inC}
	Let $q \equiv 1 \pmod 6$ be a prime power.
	Then $c^3_q(1,2)$ is even if and only if $2 \in C^3_q(0)$.
\end{lemma}

Now we present a simple condition that guarantees that $c^3_q(1,2)$ is odd.

\begin{corollary}\label{cor:l2oddCyc}
	Let $p$ be an odd prime, let $n$ be the order of $2$ modulo $p$, and let $e$ be the order of $p$ modulo $3n$.
	Suppose that $e > 1$.
	Set $q = p^a$ where $a \not \equiv 0 \pmod e$.
	Further suppose that $q \equiv 1 \pmod 6$.
	Then $c^3_q(1,2)$ is odd.
\end{corollary}

\begin{proof}
	By Lemma~\ref{lem:2inC}, it suffices to show that $2 \in C_q^3(0)$ implies that $a \equiv 0 \pmod e$.
	Suppose that $2 = \rho^{3r}$ for some $r$.
	Then $2^n = 1 = \rho^{3nr}$ and hence $3nr = q - 1$.
	Thus $q = p^a \equiv 1 \pmod {3n}$.
	We then see that $a \equiv 0 \pmod e$, as required.
\end{proof}

We will use the next (obvious) result to find primes $p$ for which the $e$ in Corollary~\ref{cor:l2oddCyc} is strictly greater than $1$.

\begin{proposition}\label{pro:mnot1}
	Let $p$ be an odd prime.
	Suppose that $p \equiv 1 \pmod{3n}$ where $n$ is the order of $2$ modulo $p$.
	Then $2^{(p-1)/3} \equiv 1 \pmod p$.
\end{proposition}

Let $p$ be a prime such that $2^{(p-1)/3} \not \equiv 1 \pmod p$.
By Proposition~\ref{pro:mnot1}, we have $p \not \equiv 1 \pmod{3n}$ where $n$ is the order of $2$ modulo $p$.
Such primes can be used to construct edge-regular graphs (via Theorem~\ref{thm:ERG2}) since the corresponding $e$ in Corollary~\ref{cor:l2oddCyc} will be greater than $1$.

\begin{remark}
	\label{rem:infinitelymanyp}
	The condition that $2^{(p-1)/3} \not \equiv 1 \pmod p$ is the same as saying that $2$ is a cubic non-residue modulo $p$.
	We remark that there exist infinitely many primes $p \equiv 1 \pmod 3$ such that $2$ is a cubic non-residue modulo $p$ (see Cox~\cite[Theorem 9.8 and Theorem 9.12]{Cox}).
\end{remark}


Let $q = p^a$ be a prime power such that $q \equiv 1 \pmod{6}$.
Then either $p \equiv 1 \pmod 3$ or $p \equiv 5 \pmod 6$ and $a$ is even.
Our next result shows that the hypotheses of Corollary~\ref{cor:l2oddCyc} can only be satisfied if $p \equiv 1 \pmod 3$.

\begin{proposition}\label{pro:p5mod6}
	Let $p \equiv 5 \pmod 6$ be a prime, let $n$ be the order of $2$ modulo $p$, and let $e$ be the order of $p$ modulo $3n$.
	Then $e = 2$.
\end{proposition}
\begin{proof}
	Let $n$ be the order of $2$ modulo $p$.
	Since $2^{(p^2-1)/3} = \left ( 2^{(p+1)/3} \right )^{p-1} \pmod {p}$, we see that $n$ divides $(p^2-1)/3$.
	Hence $p^2 \equiv 1 \pmod{3n}$ and so $e$ equals $1$ or $2$.	
	If $e = 1$ then $p \equiv 1 \pmod{3n}$, which is impossible since $p \equiv 5 \pmod 6$.
\end{proof}

%

\section{On being strongly regular} 
\label{sec:notSRGs}

Neumaier's question is about edge-regular graphs that are \emph{not} strongly regular.
In this section, we show that the graphs constructed in Section~\ref{sub:m_3} and Section~\ref{sub:_m_2} are not strongly regular (see Corollary~\ref{cor:notSRG3} and Corollary~\ref{cor:notSRG2}).

The \textbf{eigenvalues} of a graph are defined to be the eigenvalues of its adjacency matrix. 
The next proposition is a standard result about the eigenvalues of a strongly regular graph (see Brouwer and Haemers~\cite[Chapter 9]{brou:spec11} or Cameron and Van Lint~\cite[Chapter 2]{CvL}).

\begin{proposition}\label{pro:SRG}
	Let $\Gamma$ be a non-complete strongly regular graph with parameters $(N,k,\lambda,\mu)$ and eigenvalues $k > \theta_1 \geqslant \theta_2$.
	Then
	\begin{align*}
		(N-k-1)\mu &= k(k-\lambda-1); \\
		\lambda-\mu &= \theta_1+\theta_2; \\
		\mu - k &= \theta_1 \theta_2.
	\end{align*}
\end{proposition}

Let $\Gamma$ be a graph and let $\pi = \{\pi_1,\dots,\pi_s\}$ be a partition of the vertices of $\Gamma$.
For each vertex $x$ in $\pi_i$, write $d_{x}^{(j)}$ for the number of neighbours of $x$ in $\pi_j$.
Then we write $b_{ij} = 1/|\pi_i|\sum_{x \in \pi_i} d_{x}^{(j)}$ for the average number of neighbours in $\pi_j$ of vertices in $\pi_i$.
The matrix $B_\pi := (b_{ij})$ is called the \textbf{quotient matrix} of $\pi$ and $\pi$ is called \textbf{equitable} if for all $i$ and $j$, we have $d_{x}^{(j)} = b_{ij}$ for each $x \in \pi_i$.
By \cite[Theorem 9.3.3]{God01}, if $\pi$ is an equitable partition of $\Gamma$ then the eigenvalues of the quotient matrix $B_\pi$ are also eigenvalues of $\Gamma$.

\begin{lemma}\label{lem:srgparam}
	Let $\Gamma$ be a non-complete strongly regular graph with parameters $(N,k,\lambda,\mu)$ having a $1$-regular clique of order $s+1$.
	Then there exists an integer $t$ such that $(N,k,\lambda,\mu) = ((s+1)(st+1),s(t+1),s-1,t+1)$.
\end{lemma}
\begin{proof}
	Let $-t-1$ be the smallest eigenvalue of $\Gamma$ and let $\mathcal C$ be a $1$-regular clique in $\Gamma$ of order $s+1$.
	By the proof of \cite[Corollary 1.2]{Neu:regCliques} we see that $t$ is an integer and $k = s(t+1)$.
	Let $V$ denote the vertex set of $\Gamma$ and let $W \subset V$ denote the vertex set of $\mathcal C$.
	Observe that the partition $\pi = \{W, V\backslash W\}$ of $V$ is equitable with quotient matrix $ B_\pi = \left ( \begin{smallmatrix}
		s & k-s \\
		1 & k-1
	\end{smallmatrix} \right )$.
	Hence the eigenvalue $s-1$ of $B_\pi$ must also be an eigenvalue $\Gamma$.
	Therefore $\Gamma$ has distinct eigenvalues $s(t+1) > s-1 \geqslant -t-1$.
	By Proposition~\ref{pro:SRG}, the graph $\Gamma$ has parameters $((s+1)(st+1),s(t+1),s-1,t+1)$, as required.
\end{proof}

We show in the next two results that the graphs constructed in Section~\ref{sub:m_3} and Section~\ref{sub:_m_2} are not strongly regular.

\begin{corollary}\label{cor:notSRG3}
	Let $q \equiv 1 \pmod{14}$ be a prime power, let $c \equiv 1 \pmod 4$, and let $l = (3c+1)/4$.
	Suppose that $c = c^7_q(1,5)$ (resp. $c =c^7_q(1,3)$) and set $\Gamma = \operatorname{Cay}(G_{l,3,q},S(\pi))$ where $\pi = \Psi_1$ (resp. $\pi = \Psi_2$).
	Then $\Gamma$ is not strongly regular.
\end{corollary}
\begin{proof}
	Consider the vertex $v$ corresponding to the identity of $G_{l,3,q}$ and the vertex $x$ corresponding to $(0,g,\rho) \in G_{l,3,q}$ where $g = (1,0,0) \in \mathbb Z_2^3$ (resp. $g = (1,1,0) \in \mathbb Z_2^3$).
	Observe that $x$ is not adjacent to $v$.
	By Proposition~\ref{pro:mu}, the vertices $v$ and $x$ have $2+ 6c$ common neighbours.
	
	Now suppose, for a contradiction, that $\Gamma$ is strongly regular (since $q > 1$ the graph $\Gamma$ is not complete).
	By Lemma~\ref{lem:regClique}, the graph $\Gamma$ has a $1$-regular clique of order $8l = 6c+2$.
	Therefore, by Lemma~\ref{lem:srgparam}, there exists an integer $t$ such that $\Gamma$ has parameters $((6c+2)((6c+1)t+1),(6c+1)(t+1),6c,t+1)$.
	Furthermore, using Theorem~\ref{thm:ERG2}, we see that $t= (q-1)/(6c+1)$.
	
	On the other hand, since the nonadjacent vertices $v$ and $x$ have $2+ 6c$ common neighbours, we must have $t = 6c + 1$ and thence $q= (6c+1)^2+1$.
	This contradicts the fact that $q$ is odd.
\end{proof}

\begin{corollary}\label{cor:notSRG2}
	Let $q \equiv 1 \pmod{6}$ be a prime power and let $\pi : (\mathbb Z^2_2)^* \to  \mathbb Z_3$ be a bijection.
	Suppose $c = c_q^3(1,2)$ is odd and $l = (c+1)/2$.
	Then $\Gamma = \operatorname{Cay}(G_{l,2,q},S(\pi))$ is not strongly regular.
\end{corollary}
\begin{proof}
	Consider the vertex $v$ corresponding to the identity of $G_{l,2,q}$
	and the vertex $x$ corresponding to $(0,g,\rho) \in G_{l,2,q}$ where $\pi(g) = 0$.
	Observe that $x$ is not adjacent to $v$.
	By Proposition~\ref{pro:mu}, the vertices $v$ and $x$ have $2+ 2c_q^3(0,1)$ common neighbours.
	
	Now suppose, for a contradiction, that $\Gamma$ is strongly regular (since $q > 1$ the graph $\Gamma$ is not complete).
	By Lemma~\ref{lem:regClique}, the graph $\Gamma$ has a $1$-regular clique of order $4l = 2c+2$.
	Therefore, by Lemma~\ref{lem:srgparam}, there exists an integer $t$ such that $\Gamma$ has parameters $((2c+2)((2c+1)t+1),(2c+1)(t+1),2c,t+1)$.
	Furthermore, using Theorem~\ref{thm:ERG2}, we see that $t= (q-1)/(2c+1)$.
	
	On the other hand, since the nonadjacent vertices $v$ and $x$ have $2+ 2c_q^3(0,1)$ common neighbours, we must have $t = 2c_q^3(0,1) + 1$ and thence $q= (2c+1)(2c_q^3(0,1) + 1)+1$.
	This contradicts the fact that $q$ is odd.
\end{proof}

%
\section{Concluding remarks} 
\label{sec:conclusion}


	Using Corollary~\ref{cor:l2oddCyc} together with Theorem~\ref{thm:ERG2}, we obtain infinite families of edge-regular graphs having regular cliques.
	Indeed, take $p \equiv 1 \pmod{3}$ to be a prime such that $2^{(p-1)/3} \not \equiv 1 \pmod p$.
	By Remark~\ref{rem:infinitelymanyp}, we know there are infinitely many such $p$.
	%
	%
	Using Proposition~\ref{pro:mnot1}, we can set $q=p^a$ for some $a$ such that the assumptions of Corollary~\ref{cor:l2oddCyc} are satisfied.
	Then $c = c^3_q(1,2)$ is odd.
	Let $l= (c+1)/2$ and let $\pi : (\mathbb Z^2_2)^* \to  \mathbb Z_3$ be a bijection.
	By Theorem~\ref{thm:ERG2}, the graph $\operatorname{Cay}(G_{l,2,q},S(\pi))$ is an edge-regular graph with parameters $(4lq,4l-2+q,4l-2)$ having a $1$-regular clique of order $4l$.
	Furthermore, by Corollary~\ref{cor:notSRG2}, $\operatorname{Cay}(G_{l,2,q},S(\pi))$ is not strongly regular.

\begin{example}
	Let $\pi : (\mathbb Z^2_2)^* \to  \mathbb Z_3$ be a bijection.
	Set $q = 7^a$ where $a \not \equiv 0 \pmod 3$.
	The graph $X(a) := \Gamma((c^3_q(1,2) + 1)/2,2,q,\pi)$ is an edge-regular graph with parameters $(2q(c^3_q(1,2) + 1),2c^3_q(1,2)+q,2c^3_q(1,2))$ having a $1$-regular clique of order $2c^3_q(1,2)+2$.
	In particular, the graph $X(1)$ is an edge-regular graph with parameters $(28,9,2)$ having a $1$-regular clique of order $4$.
\end{example}
%

The following question naturally arises.

\textbf{Question A.} 
\label{par:question_a}
What is the minimum number of vertices for which there exists an edge-regular, but not strongly regular graph having a regular clique?

Recently, Evans et al.~\cite{EGP} discovered an edge-regular, but not strongly regular graph on $16$ vertices that has $2$-regular cliques with order $4$, and proved that, up to isomorphism, this is the unique edge-regular, but not strongly regular graph on at most $16$ vertices having a regular clique.  
The regular cliques in their graph each have order $4$.
By the following result, we see that a non-strongly-regular edge-regular graph having a regular clique must have a clique of order at least $4$.

\begin{proposition}\label{pro:cliqueBound}
	Let $\Gamma$ be an edge-regular graph having a regular clique.
	Suppose that $\Gamma$ is not strongly regular.
	Then $\Gamma$ has a regular clique of order at least $4$.
\end{proposition}
\begin{proof}
	Suppose, for a contradiction, that the largest cliques of $\Gamma$ have at most $3$ vertices.
	By \cite[Theorem 1.3]{Neu:regCliques}, it suffices to assume that $\Gamma$ has a $1$-regular clique of order $3$.
	Then, using part (iii) of \cite[Lemma 1.5]{Neu:regCliques}, each edge of $\Gamma$ is in at most one clique of $\Gamma$.
	Let $v$ be a vertex of $\Gamma$.
	The subgraph of $\Gamma$ induced on the neighbourhood $\Gamma(v)$ of $v$ must therefore be a disjoint union of $l$ edges.
	
	Take a vertex $w \not \in \Gamma(v)$.
	By \cite[Theorem 1.1]{Neu:regCliques}, every clique of order $3$ is $1$-regular.
	Hence the vertices $v$ and $w$ have exactly $l$ common neighbours, and thus $\Gamma$ is strongly regular.
	We have established a contradiction, as required.
	
	Finally, using \cite[Theorem 1.1]{Neu:regCliques}, we see that the largest cliques of $\Gamma$ are regular.
\end{proof}

By Lemma~\ref{lem:regClique}, we see that the regular cliques of the graphs  $\operatorname{Cay}(G_{l,2,q},S(\pi))$ are $1$-regular.
Moreover, the four examples of Goryainov and Shalaginov~\cite{Goryainov14:Deza} each has a $1$-regular clique and the example of Evans et al.~\cite{EGP} has a $2$-regular clique.
Hence we ask the following question.

\textbf{Question B.} 
\label{par:_textbf_question_b}
Does there exist a non-strongly-regular, edge-regular graph having a regular clique with nexus greater than $2$?

\section{Acknowledgements} 
\label{sec:acknowledgement}

We are grateful to Leonard Soicher for bringing Neumaier's question to our attention and we are grateful to Alexander Gavrilyuk for pointing out the examples of Goryainov and Shalaginov.
We also appreciate the referees' suggestions for improving an earlier draft of this paper.


%

\bibliographystyle{myplain}
\bibliography{sbib}

\begin{thebibliography}{10}

\bibitem{brou:spec11}
A.~E. Brouwer and W.~H. Haemers.
\newblock {\em Spectra of Graphs}.
\newblock Universitext, Springer, New York (2012), 2012.

\bibitem{CvL}
P.~J. Cameron and J.~H. van Lint.
\newblock {\em Designs, Graphs, Codes, and their Links}.
\newblock Cambridge University Press, 1992.

\bibitem{Cox}
D.~A. Cox.
\newblock {\em Primes of the form {$x^2 + ny^2$}: Fermat, class field theory,
  and complex multiplication}.
\newblock Pure and Applied Mathematics (Hoboken). John Wiley \& Sons, Inc.,
  Hoboken, NJ, second edition, 2013.

\bibitem{EGP}
R. J. Evans, S. Goryainov, and D. Panasenko.
\newblock The smallest {N}eumaier graph and its extension.
\newblock In preparation.

\bibitem{God01}
C. Godsil and G. Royle.
\newblock {\em Algebraic Graph Theory}, volume 207 of {\em Graduate Texts in
  Mathematics}.
\newblock Springer-Verlag, New York, 2001.

\bibitem{Goryainov14:Deza}
S. Goryainov and L. Shalaginov.
\newblock {C}ayley-{D}eza graphs with fewer than 60 vertices (in {R}ussian).
\newblock {\em Siberian Electronic Mathematical Reports}, 11(268--310), 2014.

\bibitem{MacWilliams72}
F. {MacWilliams}.
\newblock {Cyclotomic numbers, coding theory and orthogonal polynomials.}
\newblock {\em {Discrete Math.}}, 3:133--151, 1972.

\bibitem{Neu:regCliques}
A. Neumaier.
\newblock Regular cliques in graphs and special $1\frac{1}{2}$-designs.
\newblock In {\em Finite geometries and designs, Proc. 2nd Isle of Thorns Conf.
  1980}, volume~49 of {\em Lect. Note Ser.}, pages 244--259. Lond. Math. Soc.,
  1981.

\bibitem{Soi:CAB15}
L.~H. Soicher.
\newblock On cliques in edge-regular graphs.
\newblock {\em J. Algebra}, 421:260--267, 2015.

\bibitem{Storer67}
T. Storer.
\newblock {\em Cyclotomy and difference sets}.
\newblock Lectures in Advanced Mathematics. Markham Publishing Company, 1967.

\end{thebibliography}

\end{document}